\documentclass[a4paper]{amsart}
\usepackage{amssymb}
\usepackage{graphics}
\usepackage{array}
\usepackage[all]{xy}

\newtheorem{thm}{Theorem}[section]
\newtheorem{prop}[thm]{Proposition}
\newtheorem{cor}[thm]{Corollary}
\newtheorem{lem}[thm]{Lemma}

\newtheorem*{thm*}{Theorem}

\theoremstyle{definition}

\newtheorem{ex}[thm]{Example}

\theoremstyle{remark}
\newtheorem{rem}[thm]{Remark}

\newcommand{\RR}{\mathbb R}
\newcommand{\ZZ}{\mathbb Z}

\newcommand{\CC}{\mathbb C}

\newcommand{\rk}{\operatorname{rank}}

\newcommand{\id}{\operatorname{id}}

\newcommand{\Ad}{\operatorname{Ad}}

\newcommand{\SU}{\operatorname{SU}}
\newcommand{\SO}{\operatorname{SO}}
\newcommand{\Sp}{\operatorname{Sp}}
\newcommand{\PSp}{\operatorname{PSp}}

\newcommand{\mfk}{\mathfrak{k}}

\newcommand{\mfa}{\mathfrak{a}}
\newcommand{\mfg}{\mathfrak{g}}

\newcommand{\mft}{\mathfrak{t}}
\newcommand{\mfp}{\mathfrak{p}}

\newcommand{\mfe}{\mathfrak{e}}
\newcommand{\mff}{\mathfrak{f}}

\begin{document}

\title{The equivariant cohomology of isotropy actions on symmetric spaces}

\author{Oliver Goertsches}
\address{Oliver Goertsches, Mathematisches Institut, Universit\"at zu K\"oln, Weyertal 86-90, 50931 K\"oln, Germany}
\email{ogoertsc@math.uni-koeln.de}

\begin{abstract} We show that for every symmetric space $G/K$ of compact type with $K$ connected, the $K$-action on $G/K$ by left translations is equivariantly formal.
\end{abstract}
\maketitle

\section{Introduction}
Given compact connected Lie groups $K\subset G$ of equal rank, it is well-known that the $K$-action on the homogeneous space $G/K$ is equivariantly formal because the odd de Rham cohomology groups of $G/K$ vanish. (See for example \cite{GHZ} for an investigation of the equivariant cohomology of such spaces.) If however the rank of $K$ is strictly smaller than the rank of $G$, then the isotropy action is not necessarily equivariantly formal, and in general it is unclear when this is the case.\footnote{A sufficient condition for equivariant formality of the isotropy action was introduced in \cite{Shiga}, see Remark \ref{rem:Shiga} below. If $K$ belongs to a certain class of subtori of $G$ this condition is in fact an equivalence, see \cite{ShigaTakahashi}.} Restricting our attention to symmetric spaces of compact type, we will prove the following theorem.

\begin{thm*} Let $(G,K)$ be a symmetric pair of compact type, where $G$ and $K$ are compact connected Lie groups. Then the $K$-action on the symmetric space $M=G/K$ by left translations is equivariantly formal.
\end{thm*}

For symmetric spaces of type II, i.e., compact Lie groups, this result is already known, see Section \ref{sec:groups}. More generally, in the case of symmetric spaces of split rank ($\rk G=\rk K+\rk G/K$),  the fact that all $K$-isotropy groups have maximal rank implies equivariant formality, see Section \ref{sec:split}.  However, for the general case we have to rely on an explicit calculation of the dimension of the cohomology of the $T$-fixed point set $M^{T}$, where $T\subset K$ is a maximal torus, in order to use the characterization of equivariant formality via the condition $\dim H^*(M^{T})=\dim H^*(M)$. With the help of the notion of compartments  introduced in \cite{EMQ} and several results proven therein we will find in Section \ref{sec:fixed set} a calculable expression for this dimension, and after reducing to the case of an irreducible simply-connected symmetic space in Section \ref{sec:reduction} we can invoke the classification of such spaces to show equivariant formality in each of the remaining cases by hand. On the way we obtain a formula for the number of compartments in a fixed $K$-Weyl chamber, see Proposition \ref{prop:rasquotientofweylgroups}.
\\[0.3cm]
\noindent
{\bf Acknowledgements.} The author wishes to express his gratitude to Augustin-Liviu Mare for interesting discussions on a previous version of the paper.

\section{Symmetric spaces}

Let $G$ be a connected Lie group and $K\subset G$ a closed subgroup. Then $K$ is said to be a symmetric subgroup of $G$ if there is an involutive automorphism $\sigma:G\to G$ such that $K$ is an open subgroup of the fixed point subgroup $G^\sigma$. We will refer to the pair $(G,K)$ as a symmetric pair, and $G/K$ is a symmetric space.

Given a symmetric pair $(G,K)$ with corresponding involution $\sigma:G\to G$, then the Lie algebra $\mfg$ decomposes into the $(\pm 1)$-eigenspaces of $\sigma$: $\mfg=\mfk\oplus \mfp$, and the usual commutation relations hold: $[\mfk,\mfk]\subset \mfk$, $[\mfk,\mfp]\subset \mfp$ and $[\mfp,\mfp]\subset \mfk$. The rank of $G/K$ is by definition the maximal dimension of an abelian subalgebra of $\mfp$. Then clearly $\rk G-\rk K \leq \rk G/K$, and if equality holds, then we say that $G/K$ is of split rank.

A symmetric pair $(G,K)$ is called (almost) effective if $G$ acts (almost) effectively on $G/K$. Given a symmetric pair $(G,K)$, then the kernel $N\subset G$ of the $G$-action on $G/K$ is contained in $K$, and $(G/N,K/N)$ is an effective symmetric pair with $(G/N)/(K/N)=G/K$. An almost effective symmetric pair $(G,K)$ (and the corresponding symmetric space $G/K$) will be called of compact type if $G$ is a compact semisimple Lie group. In this paper only symmetric spaces of compact type will occur. If $(G,K)$ is effective, then $G$ can be regarded as a subgroup of the isometry group of $G/K$ with respect to any $G$-invariant Riemannian metric on $G/K$. If $(G,K)$ is additionally of compact type, then this inclusion is in fact an isomorphism between $G$ and the identity component of the isometry group.

\section{Equivariant formality}
The equivariant cohomology of an action of a compact connected Lie group $K$ on a compact manifold $M$ is by definition the cohomology of the Borel construction
\[
H^*_K(M)=H^*(EK\times_K M);
\]
we use real coefficients throughout the paper. The projection $EK\times_K M\to EK/K=BK$ to the classifying space $BK$ of $K$ induces on $H^*_K(M)$ the structure of an $H^*(BK)$-algebra.

An action of a compact connected Lie group $K$ on a compact manifold $M$ is called equivariantly formal in the sense of \cite{GKM} if $H^*_K(M)$ is a free $H^*(BK)$-module.  If the $K$-action on $M$ is equivariantly formal then automatically 
\begin{equation} \label{eq:eqformalequality}
H^*_K(M)=H^*(M)\otimes H^*(BK)
\end{equation} as graded $H^*(BK)$-modules, see \cite[Proposition 2.3]{GR}. In the following proposition we collect some known equivalent characterizations of equivariant formality.
\begin{prop} \label{prop:eqformalequivalent} Consider an action of a compact connected Lie group $K$ on a compact manifold $M$, and let $T\subset K$ be a maximal torus. Then the following conditions are equivalent:
\begin{enumerate}
\item The $K$-action on $M$ is equivariantly formal.
\item The $T$-action on $M$ is equivariantly formal.
\item The cohomology spectral sequence associated to the fibration $ET\times_T M\to BT$ collapses at the $E_2$-term.
\item We have $\dim H^*(M)=\dim H^*(M^T)$.
\item The natural map $H^*_T(M)\to H^*(M)$ is surjective.
\end{enumerate}
\end{prop}
\begin{proof}
For the equivalence of $(1)$ and $(2)$ see \cite[Proposition C.26]{GGK}. The Borel localization theorem implies that the rank of $H^*_T(M)$ as an $H^*(BT)$-module always equals $\dim H^*(M^T)$. Then \cite[Lemma C.24]{GGK} implies the equivalence of $(2)$, $(3)$, and $(4)$; see also \cite[p.~46]{Hsiang}.  For the equivalence to $(5)$, see \cite[p.~148]{McCleary}.
\end{proof}
Note that by \cite[p.~46]{Hsiang} the inequality $\dim H^*(M^T)\leq \dim H^*(M)$ holds for any $T$-action on $M$. Condition $(5)$ in the proposition shows that 
\begin{cor} \label{cor:subgroupseqformal}
If a compact connected Lie group $K$ acts equivariantly formally on a compact manifold $M$, then so does every connected closed subgroup of $K$.
\end{cor}

Applying the gap method to the spectral sequence in Item $(3)$ of Proposition \ref{prop:eqformalequivalent} we obtain the following well-known sufficient condition for equivariant formality.
\begin{prop} \label{prop:hoddeqformal}
Any action of a compact Lie group $K$ on a compact manifold $M$ with $H^{odd}(M)=0$ is equivariantly formal.
\end{prop}

\section{Isotropy actions on symmetric spaces of compact type}

Let $G$ be a compact connected Lie group and $K\subset G$ a compact connected subgroup. Because an equivariantly formal torus action always has fixed points, the only tori $T\subset G$ that can act equivariantly formally on $G/K$ by left translations are those that are conjugate to a subtorus of $K$.  On the other hand, if a maximal torus $T$ of $K$ acts equivariantly formally on $G/K$, then we know by Corollary \ref{cor:subgroupseqformal} that all these tori do in fact act equivariantly formally. In the following, we will prove that this indeed happens for symmetric spaces of compact type. More precisely:

\begin{thm} \label{thm:main} Let $(G,K)$ be a symmetric pair of compact type, where $G$ and $K$ are compact connected Lie groups. Then the $K$-action on the symmetric space $G/K$ by left translations is equivariantly formal.
\end{thm}

\begin{rem} \label{rem:Shiga} The pair $(G,K)$ is a Cartan pair in the sense of \cite{Greub}, see \cite[p.~448]{Greub}. Therefore,  \cite[Theorem A]{Shiga} shows that a sufficient condition for the $K$-action on $G/K$ to be equivariantly formal is that the map $H^*(G/K)^{N_G(K)}\to H^*(G)$ induced by the projection $G\to G/K$, where $N_G(K)$ acts on $G/K$ from the right, is injective. It would be interesting to know whether a symmetric pair always satisfies this condition.
\end{rem}

\subsection{The fixed point set of a maximal torus in $K$} \label{sec:fixed set}

Let $(G,K)$ be a symmetric pair of compact type, where $G$ and $K$ are compact connected Lie groups. Denote by $\sigma:G\to G$ the corresponding involutive automorphism. Then $M=G/K$ is a symmetric space of compact type. We fix maximal tori $T_K\subset K$ and $T_G\subset G$ such that $T_K\subset T_G$.  Let $\mfg=\mfk\oplus \mfp$ be the decomposition of the Lie algebra $\mfg$ into eigenspaces of $\sigma$. 

In order to prove Theorem \ref{thm:main} we can without loss of generality assume that the symmetric pair $(G,K)$ is effective: if $N\subset K$ is the kernel of the $G$-action on $G/K$, then clearly the $K$-action on $G/K=(G/N)/(K/N)$ is equivariantly formal if and only if the $K/N$-action is equivariantly formal. (This follows for example from Proposition \ref{prop:eqformalequivalent} because the fixed point sets of appropriately chosen maximal tori in $K$ and $K/N$ coincide.)

\begin{lem} \label{lem:T-fixedpoints}
The $T_K$-fixed point set in $M$ is $N_G(T_K)/N_K(T_K)$. 
\end{lem}
\begin{proof}
An element $gK\in M$ is fixed by $T_K$ if and only if $g^{-1}T_Kg\subset K$ (i.e., $g^{-1}T_K g$ is a maximal torus in the compact Lie group $K$), which is the case if and only if there is some $k\in K$ with $k^{-1}g^{-1}T_Kgk=T_K$. Thus, $(G/K)^{T_K}=N_G(T_K)/N_G(T_K)\cap K=N_G(T_K)/N_K(T_K)$.
\end{proof}

\begin{lem}[{\cite[Proposition VII.3.2]{Loos}}]\label{lem:uniquetorus} $T_G$ is the unique maximal torus in $G$ containing $T_K$. \end{lem}

Lemma \ref{lem:uniquetorus} implies that the Lie algebra $\mft_\mfg$ of $T_G$ decomposes according to the decomposition $\mfg=\mfk\oplus \mfp$ as $\mft_\mfg=\mft_\mfk\oplus \mft_\mfp$. (In fact, this statement is the first part of the proof of \cite[Proposition VII.3.2]{Loos}.)

\begin{prop} \label{prop:conncomponentsoffixedpoints} Each connected component of $M^{T_K}$ is a  torus of dimension $\rk G-\rk K$. 
\end{prop}
\begin{proof}
Because of Lemma \ref{lem:uniquetorus}, the abelian subalgebra $\mft_\mfp \subset \mfp$ is the space of elements in $\mfp$ that commute with $\mft_\mfk$. Thus, Lemma \ref{lem:T-fixedpoints} implies that the component of $M^{T_K}$ containing $eK$ is $T_G/(T_G\cap K)=T_G/T_K$ (note that the centralizer of $T_K$ in $K$ is exactly $T_K$), i.e., a $\rk G-\rk K$-dimensional torus.  Because the fixed set  $M^{T_K}$ is a homogeneous space, all components are diffeomorphic. 
\end{proof}
We therefore understand the structure of the $T_K$-fixed point set $M^{T_K}$ if we know its number of connected components, which we denote by $r$.  In view of condition $(4)$ in Proposition \ref{prop:eqformalequivalent}, we are mostly interested in the dimension of its cohomology.
\begin{prop}\label{prop:cohomoffixedpointset} We have  $\dim H^*(M^{T_K})=2^{\rk G-\rk K}\cdot r$.
\end{prop}
In order to get a calculable expression for $r$ we will use several results from \cite[Sections 5 and 6]{EMQ} which we now collect. 
Denote by $\Delta_G=\Delta_\mfg$ the root system of $G$ with respect to the maximal torus $T_G$, i.e., the set of nonzero elements $\alpha\in \mft_\mfg^*$ such that the corresponding eigenspace $\mfg_\alpha=\{X\in \mfg^\CC \mid [W,X]=i\alpha(W)X \text{ for all }W\in \mft_\mfg\}$ is nonzero. Then we have the root space decomposition
\begin{equation}\label{eq:rootspacedecomp}
\mfg^\CC=\mft_\mfg^\CC \oplus\bigoplus_{\alpha\in \Delta_\mfg} \mfg_\alpha.
\end{equation}
The $\mfg$-Weyl chambers are the connected components of the set $\mft_\mfg\setminus \bigcup_{\alpha\in \Delta_\mfg} \ker \alpha$. Because of Lemma \ref{lem:uniquetorus}, $\mft_\mfk$ contains $\mfg$-regular elements, hence no root in $\Delta_\mfg$ vanishes on $\mft_\mfk$. Therefore some of the $\mfg$-Weyl chambers intersect $\mft_\mfk$ nontrivially, and following \cite{EMQ} we will refer to these intersections as \emph{compartments}.  Considering as in \cite{EMQ} the decomposition of $\Delta_\mfg$ into complementary subsets $\Delta_\mfg=\Delta'Ê\cup \Delta''$, where
\begin{equation}\label{eq:decomprootsystem}
\Delta'=\{\alpha\in \Delta_\mfg\mid \mfg_\alpha\not\subset \mfp^{\CC}\},\quad \Delta''=\{\alpha\in \Delta_\mfg\mid \mfg_\alpha\subset \mfp^{\CC}\},
\end{equation}
we have by \cite[Lemma 9]{EMQ} that the root system $\Delta_K=\Delta_\mfk$ of $K$ with respect to $T_K$ is given by
\begin{equation}\label{eq:Krootsystem}
\Delta_\mfk=\{\left.\alpha\right|_{\mft_\mfk}\mid \alpha\in \Delta'\}.
\end{equation}
In particular, $\mfg$-regular elements in $\mft_\mfk$ are also $\mfk$-regular, and hence each compartment is contained in a $\mfk$-Weyl chamber.

Because of Lemma \ref{lem:uniquetorus}, the group $N_G(T_K)$ is a subgroup of $N_G(T_G)$. Both groups have the same identity component $T_G$, so we may regard the quotient group $N_G(T_K)/T_G$ as a subgroup of the Weyl group $W(G)$ of $G$. The free action of $W(G)$ on the $\mfg$-Weyl chambers induces an action of $N_G(T_K)/T_G$ on the set of compartments. Because any two compartments are $G$-conjugate \cite[Theorem 10]{EMQ}, this action is simply transitive on the set of compartments, and it follows that the number of connected components of $N_G(T_K)$ equals the total number of compartments in $\mft_\mfk$. On the other hand no connected component of $N_G(T_K)$ contains more than one connected component of $N_K(T_K)$. (An element in $N_K(T_K)\cap T_G$ is an element in $K$ centralizing $T_K$, hence already contained in $T_K$.)  Because the number of connected components of $N_K(T_K)$ equals the number of $\mfk$-Weyl chambers, and each $\mfk$-Weyl chamber contains the same number of compartments \cite[Theorem 10]{EMQ}, we have shown the following lemma.
\begin{lem} \label{lem:rindependentofGK}
The number $r$ of connected components of $M^{T_K}=N_G(T_K)/N_K(T_K)$ is the number of compartments in a fixed $\mfk$-Weyl chamber. In particular it only depends on the Lie algebra pair $(\mfg,\mfk)$.
\end{lem}

 Let $C$ be a $\mfg$-Weyl chamber that intersects $\mft_\mfk$ nontrivially. By \cite[Lemma 8]{EMQ} the compartment $C\cap \mft_\mfk$ can be described explicitly: The involution $\sigma:G\to G$ permutes the $\mfg$-Weyl chambers and fixes $\mft_\mfk$, hence it fixes $C$. Let $B=\{\alpha_1,\ldots,\alpha_{\rk G}\}$ be the corresponding simple roots such that $C$ is exactly the set of points where the elements of $B$ take positive values.  The involution $\sigma$ acts as a permutation group on $B$ because for any $i$ the linear form $\alpha_i\circ \sigma$ is again positive on $C$. Note that for every root $\alpha\in \Delta_\mfg$ the linear form $\frac{1}{2}(\alpha + \alpha \circ \sigma)$ vanishes on $\mft_\mfp$ and coincides with $\left.{\alpha}\right|_{\mft_\mfk}$ on $\mft_\mfk$. The set $\left.B\right|_{\mft_\mfk}=\{\left.{\alpha_i}\right|_{\mft_\mfk}\mid i=1,\ldots,\rk G\}$ is a basis of $\mft_\mfk^*$ (in particular it consists of $\dim \mft_\mfk$ elements)  and the compartment $C\cap \mft_\mfk$ is exactly the set of points in $\mft_\mfk$ where all $\left.{\alpha_i}\right|_{\mft_\mfk}$ take positive values. It is a simplicial cone bounded by the hyperplanes $\ker \left.{\alpha_i}\right|_{\mft_\mfk}$. Any such hyperplane is either a wall of a $\mfk$-Weyl chamber or the kernel of a $\mfg$-root $\alpha_i$ with $\alpha_i\circ \sigma=\alpha_i$, see \eqref{eq:Krootsystem}. In any case, reflection along the hyperplane defines an element of $N_G(T_K)/T_G$ and takes  $C\cap \mft_\mfk$ to an adjacent compartment. (This argument is taken from the proof of \cite[Theorem 10]{EMQ}.) 
 
It follows that the action of $N_G(T_K)/T_G$ on the set of compartments described above is generated by the reflections along all hyperplanes $\ker \left.\alpha\right|_{\mft_\mfk}$, where $\alpha\in \Delta_\mfg$. Let $\langle\cdot,\cdot\rangle$ be the Killing form on $\mfg$. The decomposition $\mfg=\mfk\oplus \mfp$ is orthogonal with respect to $\langle\cdot,\cdot\rangle$. We identify $\mft_\mfg^*$ with $\mft_\mfg$ and $\mft_\mfk^*$ with $\mft_\mfk$ via $\langle\cdot,\cdot\rangle$. For $\alpha\in \Delta_\mfg$, let $H_\alpha\in \mft_\mfg$ be the element such that $\alpha(H)=\langle H,H_\alpha\rangle$ for all $H\in \mft_\mfg$. Given $X\in \mft_\mfg$, we write $X^\mfk$ and $X^\mfp$ for the $\mfk$- and $\mfp$-parts of $X$ respectively. Then $H_\alpha^\mfk$ corresponds to $\left.\alpha\right|_{\mft_\mfk}$ under the isomorphism $\mft_\mfk\cong \mft_\mfk^*$.
\begin{lem}\label{lem:rootlemma} Let $\alpha\in \Delta_\mfg$ be a root with $\alpha\circ \sigma\neq \alpha$. Then either
\begin{enumerate}
\item $\langle H_\alpha,H_{\alpha\circ \sigma}\rangle=0$ and $|H^\mfp_\alpha|^2=|H^\mfk_\alpha|^2$ or
\item $2\cdot\frac{\langle H_\alpha,H_{\alpha\circ \sigma}\rangle}{|H_\alpha|^2}=-1$, $|H^\mfp_\alpha|^2=3|H^\mfk_\alpha|^2$ and $\alpha+\alpha\circ \sigma\in \Delta_\mfg$.
\end{enumerate}
\end{lem}
\begin{proof} We have $H_{\alpha\circ \sigma}=H^\mfk_\alpha-H^\mfp_\alpha$, and because $\Delta_\mfg$ is a root system it follows that
\[
2\cdot\frac{\langle H_\alpha,H_{\alpha\circ \sigma}\rangle}{|H_\alpha|^2}=2\cdot \frac{|H_\alpha^\mfk|^2 - |H_\alpha^\mfp|^2}{|H_\alpha^\mfk|^2+ |H_\alpha^\mfp|^2}\in \ZZ.
\]
Because $\alpha$ and $\alpha\circ \sigma$ are roots of equal length, this integer can only equal $0$ or $\pm 1$ \cite[Proposition 2.48.(d)]{Knapp}. Further, because $\alpha-\alpha\circ \sigma$ is not a root (by Lemma \ref{lem:uniquetorus} no root vanishes on $\mft_\mfk$) and not $0$, only the possibilities $0$ and $-1$ remain, and in the latter case we also have that $\alpha+\alpha\circ\sigma\in \Delta_\mfg$ \cite[Proposition 2.48.(e)]{Knapp}. 
\end{proof}
\begin{prop}  The set $\left.\Delta_\mfg\right|_{\mft_\mfk}=\{\left.\alpha\right|_{\mft_\mfk}\mid \alpha\in \Delta_\mfg\}$ is a root system in $\mft_\mfk^*$.
\end{prop}
\begin{proof} 
It is clear that $\left.\Delta_\mfg\right|_{\mft_\mfk}$ spans $\mft_\mfk^*$. We have to check that for all $\alpha,\beta\in \Delta_\mfg$, the quantity
\begin{equation}\label{eq:checkinteger}
2\cdot \frac{\langle H_\alpha^\mfk,H_\beta^\mfk\rangle}{|H_\alpha^\mfk|^2}
\end{equation}
is an integer. With respect to the decomposition $\Delta_\mfg=\Delta'\cup \Delta''$ (see \eqref{eq:decomprootsystem}) there are four cases:

If both $\alpha$ and $\beta$ are elements of $\Delta'$, then \eqref{eq:checkinteger} is an integer because $\left.\alpha\right|_{\mft_\mfk}$ and $\left.\beta\right|_{\mft_\mfk}$ are $\mfk$-roots, see \eqref{eq:Krootsystem}. In case $\alpha$  and $\beta$ are elements of $\Delta''$, then the corresponding vectors $H_\alpha$ and $H_\beta$ are already elements of $\mft_\mfk$, so $H_\alpha^\mfk=H_\alpha$ and $H_\beta^\mfk=H_\beta$, hence \eqref{eq:checkinteger} is an integer.

Consider the case that $\alpha\in \Delta''$ and $\beta\in \Delta'$. Then $H_\alpha=H_\alpha^\mfk\in \mft_\mfk$, hence
\[
2\cdot \frac{\langle H_\alpha^\mfk,H_\beta^\mfk\rangle}{|H_\alpha^\mfk|^2}= 2\cdot \frac{\langle H_\alpha,H_\beta\rangle}{|H_\alpha|^2}\in \ZZ.
\]
The last case to be considered is that $\alpha\in \Delta'$ and $\beta\in \Delta''$. In this case $H_\beta=H_\beta^\mfk\in \mft_\mfk$. It may happen that $H_\alpha\in \mft_\mfk$, but then the claim would follow as before, so we may assume that $H_\alpha\notin \mft_\mfk$. It follows that $\alpha\circ \sigma$ is a root different from $\alpha$. By Lemma \ref{lem:rootlemma} we have $|H^\mfp_\alpha|^2=c|H^\mfk_\alpha|^2$ with $c=1$ or $c=3$.  We know that
 \[
 2\cdot \frac{\langle H_\alpha,H_\beta \rangle}{|H_\alpha|^2} =  2\cdot \frac{\langle H_\alpha^\mfk,H_\beta^\mfk \rangle}{|H_\alpha^\mfk|^2+|H_\alpha^\mfp|^2}=\frac{2}{1+c}\cdot \frac{\langle H_\alpha^\mfk,H_\beta^\mfk \rangle}{|H_\alpha^\mfk|^2}
 \]
 is an integer, hence multiplying with the integer $1+c$ shows that  \eqref{eq:checkinteger} is an integer in this case as well.

Next we have to check that for each $\alpha\in \Delta_\mfg$ the reflection $s_{\left.\alpha\right|_{\mft_\mfk}}:\mft_\mfk\to \mft_\mfk$ along $\ker \left.\alpha\right|_{\mft_\mfk}$ defined by
\begin{equation}\label{eq:reflection}
X\mapsto X-2\cdot \frac{\langle H_\alpha^\mfk,X\rangle}{|H_\alpha^\mfk|^2} H_\alpha^\mfk
\end{equation}
sends $\{H_\beta^\mfk\mid \beta\in \Delta_\mfg\}$ to itself. If $H_\alpha\in \mft_\mfk$ (this includes the case $\alpha\in \Delta''$), then the reflection $s_\alpha:\mft_\mfg\to \mft_\mfg$ along $\ker \alpha$ leaves invariant $\mft_\mfk$, and \eqref{eq:reflection} is nothing but the restriction of this reflection to $\mft_\mfk$. Thus, $\{H_\beta^\mfk\mid \beta\in \Delta_\mfg\}$ is sent to itself.

Let $\alpha\in \Delta'$ with $H_\alpha\notin \mft_\mfk$. 
We treat the two cases that can arise by Lemma \ref{lem:rootlemma} separately: assume first that $\langle H_\alpha,H_{\alpha\circ \sigma}\rangle=0$. In this case the two reflections $s_\alpha$ and $s_{\alpha\circ \sigma}$ commute and we have, recalling that $H_{\alpha\circ \sigma}=H_\alpha^\mfk-H_\alpha^\mfp$,
\begin{align*}
s_{\alpha\circ \sigma}\circ s_\alpha(X)&=X-2\cdot \frac{\langle H_\alpha,X\rangle}{|H_\alpha|^2} H_\alpha-2\cdot \frac{\langle H_{\alpha\circ \sigma},X\rangle}{|H_{\alpha\circ \sigma}|^2} H_{\alpha\circ \sigma}\\
&=X-2\cdot \frac{\langle H_{\alpha},X\rangle+\langle H_{\alpha\circ \sigma},X\rangle}{2|H_\alpha^\mfk|^2} H_\alpha^\mfk
- 2\cdot \frac{\langle H_{\alpha},X\rangle-\langle H_{\alpha\circ \sigma},X\rangle}{2|H_\alpha^\mfp|^2} H_\alpha^\mfp\\
&=X-2\cdot \frac{\langle H_\alpha^\mfk,X\rangle}{|H_\alpha^\mfk|^2}H_\alpha^\mfk + 2 \cdot \frac{\langle H_\alpha^\mfp,X\rangle}{|H_\alpha^\mfp|^2}H_\alpha^\mfp.
\end{align*} 
In particular for each $\beta\in \Delta_\mfg$ the vector $H_\beta^\mfk-2\cdot \frac{\langle H_\alpha^\mfk,H_\beta^\mfk\rangle}{|H_\alpha^\mfk|^2}H_\alpha^\mfk$ is the $\mfk$-part of some vector $H_\gamma$, which shows that \eqref{eq:reflection} sends $\{H_\beta^\mfk\mid \beta\in \Delta_\mfg\}$ to itself.

In the second case of Lemma \ref{lem:rootlemma} we have that $\alpha+\alpha\circ \sigma\in \Delta_\mfg$, with $\ker (\alpha+\alpha\circ \sigma)=\ker \left.\alpha\right|_{\mft_\mfk} \oplus \mft_\mfp$. Thus, the reflection $s_{\left.\alpha\right|_{\mft_\mfk}}$ is nothing but the restriction of $s_{\alpha+\alpha\circ\sigma}$ to $\mft_\mfk$; in particular it sends $\{H_\beta^\mfk\mid \beta\in \Delta_\mfg\}$ to itself.
\end{proof}
\begin{rem} \label{rem:reduced}
The root system $\left.\Delta_\mfg\right|_{\mft_\mfk}$ is not necessarily reduced: if there exists a root $\alpha\in \Delta_\mfg$ with $\alpha\circ \sigma\neq \alpha$ for which the second case of Lemma \ref{lem:rootlemma} holds, then it contains $\left.\alpha\right|_{\mft_\mfk}$ as well as $2 \cdot\left.\alpha\right|_{\mft_\mfk}$. This happens for instance for $\SU(2m+1)/\SO(2m+1)$.
\end{rem}

Because $B$ is the set of simple roots of $\Delta_\mfg$ every root $\alpha\in \Delta_\mfg$ can be written as a linear combination of elements in $B$ with integer coefficients of the same sign. It follows that every restriction $\left.\alpha\right|_{\mft_\mfk}\in \left.\Delta\right|_{\mft_\mfk}$ is a linear combination of elements in $\left. B\right|_{\mft_\mfk}$ of the same kind. We thus have proven the following lemma.
\begin{lem}
The $\left.\Delta_\mfg\right|_{\mft_\mfk}$-Weyl chambers are exactly the compartments. If $C$ is a $\mfg$-Weyl chamber that intersects $\mft_\mfk$ nontrivially, with corresponding set of simple roots $B\subset \Delta_\mfg$, then $\left.B\right|_{\mft_\mfk}$ is the set of simple roots of the root system $\left.\Delta_\mfg\right|_{\mft_\mfk}$ corresponding to $C\cap \mft_\mfk$.
\end{lem}
Recall that  the $N_G(T_K)/T_G$-action on the set of compartments was shown to be generated by the reflections along all hyperplanes $\ker \left.\alpha\right|_{\mft_\mfk}$, where $\alpha\in \Delta_\mfg$. Thus, we obtain
\begin{cor} The $N_G(T_K)/T_G$-action on the set of compartments is the same as the action of the Weyl group $W(\left.\Delta_\mfg\right|_{\mft_\mfk})$. In particular, it is generated by the reflections along the hyperplanes $\ker \left.\alpha_i\right|_{\mft_\mfk}$. Furthermore, $r=\frac{|W(\left.\Delta_\mfg\right|_{\mft_\mfk})|}{|W(\mfk)|}$.
\end{cor}

Recall  that whereas a reduced root system is determined by its simple roots \cite[Proposition 2.66]{Knapp}, this is no longer true for nonreduced root systems such as $\left.\Delta_\mfg\right|_{\mft_\mfk}$, see \cite[II.8]{Knapp}. However, the reduced elements in a nonreduced root system always form a reduced root system \cite[Lemma 2.91]{Knapp} with the same simple roots and the same Weyl group. Using the following proposition taken from \cite{Loos} we will identify this reduced root system contained in $\left.\Delta_\mfg\right|_{\mft_\mfk}$ with the root system of a second symmetric subalgebra $\mfk'\subset \mfg$. 

\begin{prop}[{\cite[Proposition VII.3.4]{Loos}}] \label{prop:loos} There is an extension of $\sigma:\mft_\mfg\to \mft_\mfg$ to an involutive automorphism $\sigma':\mfg\to \mfg$ such that its $\CC$-linear extension $\sigma':\mfg^\CC\to \mfg^\CC$ satisfies $\left.\sigma'\right|_{\mfg_{\alpha}}=\id$ for every root $\alpha\in B$ with $\alpha=\alpha\circ \sigma$.  The root system of the fixed point algebra $\mfk'=\mfg^{\sigma'}$ relative to the maximal abelian subalgebra $\mft_\mfk$ has $\left.B\right|_{\mft_\mfk}$ as simple roots.
\end{prop}

The roots of $\mfk'$ relative to $\mft_\mfk$ are restrictions of certain (not necessarily all) elements in $\Delta_\mfg$ to $\mft_\mfk$; the restrictions of all elements in $B$ occur. See \cite[p.~129]{Loos} for the root space decomposition of $\mfk'$ with respect to $\mfk_\mfk$. Because the sub-root system of reduced elements in $\left.\Delta_\mfg\right|_{\mft_\mfk}$ and the root system of $\mfk'$ have the same simple roots, these reduced root systems coincide. In particular we obtain the following formula for $r$: 

\begin{prop} \label{prop:rasquotientofweylgroups} We have $r=\frac{|W(\mfk')|}{|W(\mfk)|}$.
\end{prop}

\begin{ex} If $\rk G=\rk K$, i.e., if $T_K$ is also a maximal torus of $G$, then the identity on $\mfg$ satisfies the conditions of Proposition \ref{prop:loos}. Hence $\mfk'=\mfg$ and the proposition says $r=\frac{|W(G)|}{|W(K)|}$. This however follows already from Lemma \ref{lem:T-fixedpoints}.
\end{ex}

\begin{ex} \label{ex:splitrank} If $G/K$ is a symmetric space of split rank, i.e., $\rk G=\rk K+\rk G/K$, then $\sigma$ itself satisfies the conditions of Proposition \ref{prop:loos}. In fact, let $\alpha\in B$ with $\alpha=\alpha\circ \sigma$. In this case $\alpha$ vanishes on $\mft_\mfp$, which implies that $\mfg_\alpha$ is contained either in $\mfk^\CC$ or in $\mfp^\CC$. But if it was contained in $\mfp^\CC$, then $[\mft_\mfp,\mfg_\alpha]=0$ and $[\mft_\mfp,\mfg_{-\alpha}]=0$, which would contradict the fact that $\mft_\mfp$ is maximal abelian in $\mfp$. Thus, we have  $r=1$ in the split rank case. Note that $r=1$ also follows from \cite[Lemma 13]{EMQ}, combined with Lemma \ref{lem:rindependentofGK}.
\end{ex}

\begin{ex} The symmetric space $G/K'$, where $K'$ is the connected subgroup of $G$ with Lie algebra $\mfk'$, is not always of split rank. Assume as in Remark \ref{rem:reduced} that there exists a root $\alpha\in \Delta_\mfg$ with $\alpha\circ \sigma\neq \alpha$ such that $\alpha+\alpha\circ \sigma\in \Delta_\mfg$. Let $X\in \mfg_\alpha$ be nonzero. Then $[X,\sigma'(X)]$ is a nonzero element in $\mfg_{\alpha+\alpha\circ \sigma}$. We have $\sigma'([X,\sigma'(X)])=-[X,\sigma'(X)]$, thus $[X,\sigma'(X)]\in \mfp'$, where $\mfp'$ is the $-1$-eigenspace of $\sigma'$.  By definition of $\sigma'$ we have $\mft_\mfp\subset \mfp'$, but $\mft_\mfp$ is not a maximal abelian subspace of $\mfp'$ because it commutes with $[X,\sigma'(X)]$. For example, in the case $\SU(2m+1)/\SO(2m+1)$ we have $K'=K$ although the space is not of split rank, see Subsection \ref{sssec:su2m+1} below.
\end{ex}

We will use below that the symmetric subalgebra $\mfk'$ can be determined via the Dynkin diagram of $G$: $\sigma$ defines an automorphism of the Dynkin diagram of $G$ (because it is a permutation group of $B$), which is nontrivial if and only if $\rk \mfg >\rk \mfk$. One can calculate the root system of $\mfk'$ via the fact that by Proposition \ref{prop:loos} the simple roots of $\mfk'$ are given by  $\left.B\right|_{\mft_\mfk}=\{\frac{1}{2}(\alpha_i+\alpha_i\circ \sigma)\mid i=1,\ldots,\rk G\}$.

\subsection{Reduction to the irreducible case}\label{sec:reduction}

\begin{lem} \label{lem:eqformalindependentofgroups} If $(G,K)$ and $(G',K')$ are two effective symmetric pairs of connected compact semisimple Lie groups associated to the same pair of Lie algebras $(\mfg,\mfk)$, then the $K$-action on $G/K$ is equivariantly formal if and only if the $K'$-action on $G'/K'$ is equivariantly formal.
\end{lem}
\begin{proof} Because $K$ and $K'$ are connected, both $H^*(G/K)$ and $H^*(G'/K')$ are given as the $\RR$-algebra of $\mfk$-invariant elements in $\Lambda^* \mfp$, see \cite[Theorem 8.5.8]{Wolf}. In particular $\dim H^*(G/K)=\dim H^*(G'/K')$. Choosing maximal tori $T\subset K$ and $T'\subset K'$, we furthermore know from Propositions \ref{prop:cohomoffixedpointset} and \ref{lem:rindependentofGK} that $\dim H^*((G/K)^T)=\dim H^*((G'/K')^{T'})$ because $(G,K)$ and $(G',K')$ correspond to the same Lie algebra pair.  The statement then follows from Proposition \ref{prop:eqformalequivalent}.
\end{proof}
\begin{lem} \label{lem:eqformalproduct}
Given actions of compact connected Lie groups $K_i$ on compact manifolds $M_i$ ($i=1\ldots n$), then the $K_1\times\ldots \times K_n$-action on $M_1\times \ldots \times M_n$ is equivariantly formal if and only if all the $K_i$-actions on $M_i$ are equivariantly formal.
\end{lem}
\begin{proof} Choose maximal tori $T_i\subset K_i$. Then $T_1\times \ldots \times T_n$ is a maximal torus in $K_1\times\ldots \times K_n$. The claim follows from Proposition \ref{prop:eqformalequivalent} because the $T_1\times \ldots \times T_n$-fixed point set is exactly the product of the $T_i$-fixed point sets.
\end{proof}

Lemmas \ref{lem:eqformalindependentofgroups} and \ref{lem:eqformalproduct} imply that for proving Theorem \ref{thm:main} it suffices to check it for effective symmetric pairs $(G,K)$ of compact connected Lie groups such that $G/K$ is an irreducible simply-connected symmetric space of compact type. Below we will make use of the classification of such spaces, see \cite{Helgason}.

\subsection{Lie groups}\label{sec:groups}
Given a compact connected Lie group $G$, the product $G\times G$ acts on $G$ via $(g_1,g_2)\cdot g=g_1gg_2^{-1}$. The isotropy group of the identity element is the diagonal $D(G)\subset G\times G$. In the language of Helgason \cite{Helgason}, we obtain an  irreducible  symmetric pair $(G\times G,D(G))$ of type II. The $D(G)$-action on $(G\times G)/D(G)$ is nothing but the action of $G$ on itself by conjugation. But for any compact connected Lie group, the action on itself by conjugation is equivariantly formal. In fact, if $T\subset G$ is a maximal torus, then the fixed point set of the $T$-action, $G^T$, is $T$ itself, and thus $\dim H^*(G^T)=\dim H^*(T)=2^{\rk G}=\dim H^*(G)$. For other ways to prove that this action is equivariantly formal see \cite[Example 4.6]{GR}. For instance, equivariant formality would also follow from Proposition \ref{prop:splitrankequivformal} below as $(G\times G,D(G))$ is of split rank.

\subsection{Inner symmetric spaces}\label{sec:inner}

Consider the case that the symmetric space $G/K$ of compact type is inner, i.e., that the involution $\sigma$ is inner. By \cite[Theorem IX.5.6]{Helgason} this is the case if and only if $\rk G=\rk K$.  Hence, a maximal torus $T_K\subset K$ is also a maximal torus in $G$, and the $T_K$-fixed point set is by Lemma \ref{lem:T-fixedpoints} a finite set of cardinality $\frac{|W(G)|}{|W(K)|}$. 
Because of the following classical result (see for example \cite[Chapter XI, Theorem VII]{Greub}), the case of inner symmetric spaces is easy to deal with.
\begin{prop} Given any compact connected Lie groups $K\subset G$, the following conditions are equivalent:
\begin{enumerate}
\item $\rk G=\rk K$.
\item $\chi(G/K)>0$.
\item $H^{odd}(G/K)=0$.
\end{enumerate}
\end{prop}
It follows from Proposition \ref{prop:hoddeqformal} that the $K$-action on a homogeneous space $G/K$ with $\rk G=\rk K$ is always equivariantly formal. Alternatively, \cite[Corollary 4.5]{GR} implies that the $G$-action on $G/K$ is equivariantly formal because all its isotropy groups have rank equal to the rank of $G$. Then by Corollary \ref{cor:subgroupseqformal} any closed subgroup of $G$ acts equivariantly formally on $G/K$. 

\begin{prop} If $\rk G=\rk K$, then the $K$-action on $G/K$ is equivariantly formal. If $T_K\subset K$ is a maximal torus, then the fixed point set of the induced $T_K$-action consists of exactly $\dim H^*(G/K)=\frac{|W(G)|}{|W(K)|}$ points.
\end{prop}
\begin{rem} This is not a new result. For an investigation of the (algebra structure of the) equivariant cohomology of homogeneous spaces $G/K$ with $\rk G =\rk K$ see \cite{GHZ}, or \cite[Section 5]{HolmSjamaar} for an emphasis on other coefficient rings.
\end{rem}

\subsection{Spaces of split rank}\label{sec:split}

Also when $G/K$ is of split rank, i.e., $\rk G=\rk K + \rk G/K$, there is a general argument that implies equivariant formality of the $K$-action on $G/K$.
\begin{prop} \label{prop:splitrankequivformal} If $G/K$ is of split rank, then the natural $K$-action on $G/K$ is equivariantly formal.
\end{prop}
\begin{proof} We will show that every $K$-isotropy algebra  has maximal rank, i.e., rank equal to $\rk \mfk$. Then equivariant formality follows from \cite[Corollary 4.5]{GR}. 

Consider the decomposition $\mfg=\mfk\oplus \mfp$ and choose any $\Ad_K$-invariant scalar product on $\mfp$ that turns $G/K$ into a Riemannian symmetric space. Then we have an exponential map $\exp:\mfp\to G/K$, and it is known that every orbit of the $K$-action on $G/K$ meets $\exp(\mfa)$, where $\mfa$ is a maximal abelian subalgebra of $\mfp$. Because $G/K$ is of split rank, there is a maximal torus $T_K\subset K$ such that $\mft_\mfk \oplus \mfa$ is abelian. The torus $T_K$ acts trivially on $\exp(\mfa)$. Thus, the $K$-isotropy algebra of any point in $\exp(\mfa)$ (and hence of any point in $M$) has maximal rank.
\end{proof}

In the split-rank case we have $r=1$ by Example \ref{ex:splitrank}. We thus have

\begin{prop} If $G/K$ is of split rank then $\dim H^*(G/K)=2^{\rk G/K}$. If $T_K\subset K$ is a maximal torus, then the fixed point set of the induced $T_K$-action on $G/K$ is a $\rk G/K$-dimensional torus (in particular connected).
\end{prop}

\subsection{Outer symmetric spaces which are not of split rank}
For the remaining cases that are not covered by any of the arguments above, i.e., irreducible simply-connected symmetric spaces of type I that are neither of equal nor of split rank, we do not have a general argument for equivariant formality of the isotropy action. Using the classification of symmetric spaces \cite[p.~518]{Helgason}, we calculate for each of these spaces the dimension of the cohomology of the $T_K$-fixed point set and show that it coincides with the dimension of the cohomology of $G/K$ (which we take from the literature), upon which we conclude equivariant formality via Proposition \ref{prop:eqformalequivalent}. Fortunately, there are only three (series of) such symmetric spaces, namely
\[
\SU(n)/\SO(n),\quad \SO(2p+2q+2)/\SO(2p+1)\times \SO(2q+1), \text{ and }  E_6/\PSp(4),
\]
where $n\geq 4$ and $p,q\geq 1$. We have shown with Propositions \ref{prop:cohomoffixedpointset} and \ref{prop:rasquotientofweylgroups} that
\[
\dim H^*((G/K)^{T_K})=2^{\rk \mfg-\rk \mfk} \cdot \frac{|W(\mfk')|}{|W(\mfk)|},
\]
where the symmetric subalgebra $\mfk'\subset \mfg$ was introduced in Proposition \ref{prop:loos}. Because in this section we are dealing with outer symmetric spaces, we have $\rk \mfg>\rk \mfk$, so $\mfk'\neq \mfg$ is a symmetric subgroup of $\mfg$. The orders of the appearing Weyl groups are listed in \cite[p.~66]{Humphreys}.

\subsubsection{$\SU(2m)/\SO(2m)$} Let $M=\SU(2m)/\SO(2m)$, where $m\geq 2$, and $T\subset \SO(2m)$ be a maximal torus. The only connected symmetric subgroup of $\SU(2m)$ of rank $m$ different from $\SO(2m)$ is $\Sp(m)$. The fact that $\mfk'={\mathfrak{sp}}(m)$ can be visualized via the Dynkin diagrams: the involution $\sigma$ fixes only the middle root of the Dynkin diagram $A_{2m-1}$ of $\SU(2m)$. Hence, after restricting, the middle root becomes a root which is longer than the other roots, and only in $C_m$ there exists a root longer than the others, not in $D_m$.
\begin{center}
\includegraphics{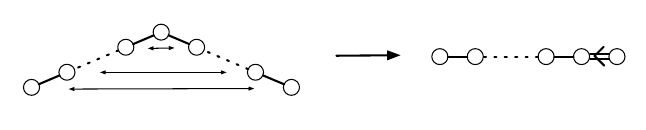} 
\end{center}
We thus may calculate
\[
r= \frac{|W(C_m)|}{|W(D_m)|} =\frac{2^m\cdot m!}{2^{m-1}\cdot m!}=2;
\]
note that for this example the number of compartments was also calculated in \cite[p.~11]{EMQ}.
It is known that $\dim H^*(M)=2^m$ (see for example \cite[p.~493]{Greub} or \cite[Theorem III.6.7.(2)]{Mimura}), hence
\[
\dim H^*(M^T)=2^{2m-1-m}\cdot r =2^m =\dim H^*(M).
\]
Thus, the action is equivariantly formal.

\subsubsection{$\SU(2m+1)/\SO(2m+1)$} \label{sssec:su2m+1} Let $M=\SU(2m+1)/\SO(2m+1)$, where $m\geq 2$, and $T\subset \SO(2m+1)$ be a maximal torus. It is known that $\dim H^*(M)=2^m$ (see for example \cite[p.~493]{Greub} or \cite[Theorem III.6.7.(2)]{Mimura}), hence
\[
2^m\cdot r = \dim H^*(M^T)\leq \dim H^*(M)=2^m
\]
for some natural number $r$. Thus necessarily $r=1$ (in fact $\mfk'={\mathfrak{so}}(2m+1)$) and the action is equivariantly formal. Note that this space is also listed as an exception in \cite{EMQ} as it is the only outer symmetric space which is not of split rank such that the corresponding involution fixes no root in the Dynkin diagram (and hence every compartment is a $K$-Weyl chamber).

\subsubsection{$\SO(2p+2q+2)/\SO(2p+1)\times \SO(2q+1)$} Let $M=\SO(2p+2q+2)/\SO(2p+1)\times \SO(2q+1)$, where $p,q\geq 1$, and $T\subset \SO(2p+1)\times \SO(2q+1)$ be a maximal torus. The only connected symmetric subgroups of $\SO(2p+2q+2)$ of rank $p+q$ are $\SO(2p'+1)\times \SO(2q'+1)$, where $p'+q'=p+q$. The involution $\sigma$ fixes all roots of the Dynkin diagram  $D_{p+q+1}$ of $\SO(2p+2q+2)$ but two; after restricting, these two become a single root which is shorter than the others. Because $A_{p+q-1}\oplus A_1$ and $D_{p+q}$ do not appear as the Dynkin diagram of any of the possible symmetric subgroups, the Dynkin diagram of $\mfk'$ is forced to be $B_{p+q}$, which means that $\mfk'={\mathfrak{so}}(2p+2q+1)$.
\begin{center}
\includegraphics{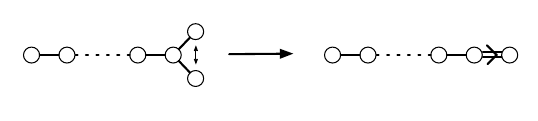} 
\end{center}
We thus have 
\[
r=\frac{|W(B_{p+q})|}{|W(B_p)|\cdot |W(B_q)|}=\frac{2^{p+q}\cdot (p+q)!}{2^p \cdot p!\cdot 2^q\cdot q!} = {p+q \choose p}.
\]
By \cite[p.~496]{Greub} we have $\dim H^*(M)=2\cdot {p+q\choose p}$, and it follows that the action is equivariantly formal because of
\begin{align*}
\dim H^*(M^T)&=2^{p+q+1-p-q}\cdot r =2\cdot {p+q \choose p} = \dim H^*(M).
\end{align*}
\subsubsection{$E_6/\PSp(4)$} Let $M=E_6/\PSp(4)$ and $T\subset \PSp(4)$ be a maximal torus. The only symmetric subalgebra of $\mfe_6$ of rank $4$ different from ${\mathfrak{sp}}(4)$ is $\mff_4$.
\begin{center}
\includegraphics{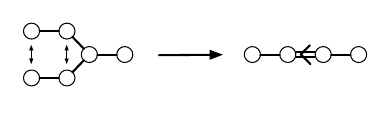} 
\end{center}
We obtain
\[
r=\frac{|W(F_4)|}{|W(C_4)|}=\frac{2^7 \cdot 3^2}{2^4 \cdot 4!}=3.
\]
It is shown in \cite{Takeuchi} that $\dim H^*(M)=12$. Thus, 
\[
\dim H^*(M^T)=2^{6-4}\cdot r = 2^2\cdot 3 = 12 = \dim H^*(M)
\]
shows that the action is equivariantly formal.

\end{document}